\documentclass{amsart}
\usepackage{amssymb,latexsym}
\theoremstyle{plain}
\newtheorem{theorem}{Theorem}
\newtheorem{corollary}{Corollary}
\newtheorem{proposition}{Proposition}
\newtheorem{lemma}{Lemma}
\theoremstyle{definition}
\newtheorem{definition}{Definition}

\newtheorem{notation}{Notation}

\begin{document}

\title[$X$-ranks]
{On the stratification by $X$-ranks of a linearly normal elliptic curve $X\subset \mathbb {P}^n$}
\author{Edoardo Ballico}
\address{Dept. of Mathematics\\
 University of Trento\\
38123 Povo (TN), Italy}
\email{ballico@science.unitn.it}
\thanks{The author was partially supported by MIUR and GNSAGA of INdAM (Italy).}

\begin{abstract}
Let $X\subset \mathbb {P}^n$ be a linearly normal elliptic curve. For any $P\in \mathbb {P}^n$
the $X$-rank of $P$ is the minimal cardinality of a set $S\subset X$ such that $P\in \langle S\rangle$.
In this paper we give an almost complete description of the stratification of $\mathbb {P}^n$ given by the $X$-rank
and the open $X$-rank. 
\end{abstract}

\maketitle

Fix an integral and non-degenerate variety $X\subset \mathbb {P}^n$. For any $P\in \mathbb {P}^n$ the $X$-rank $r_X(P)$ of $P$ is the minimal cardinality
of a subset $S\subset X$ such that $P\in \langle S\rangle$, where $\langle \ \ \rangle$ denote the linear span. The $X$-rank is an extensively studied topic
(\cite{lt}, \cite{bgl}, \cite{bgi}, \cite{l} and references therein). In the applications one needs only the cases in which
$X$ is either a Veronese embedding of a projective space or a Segre embedding of a multiprojective space. We feel that the general case gives a treasure of new projective geometry. Up to now only for
rational normal curves there is a complete description of the stratification of $\mathbb {P}^n$ by $X$-rank (\cite{cs}, \cite{lt}, Theorem 5.1, \cite{bgi}). Here we look at the case of elliptic linearly normal curves. For any integer $t \ge 1$ let $\sigma _t(X)$ denote the closure
in $\mathbb {P}^n$ of all $(t-1)$-dimensional linear spaces spanned by $t$ points of $X$. Set $\sigma _0(X)=\emptyset$.
For any $P\in \mathbb {P}^n$ the border $X$-rank $b_X(P)$ is the minimal integer $t\ge 1$ such that $P\in \sigma _t(X)$, i.e. the only positive
integer $t$ such that $P\in \sigma _t(X)\setminus \sigma _{t-1}(X)$. If (as always in this paper) $X$ is a curve,
then $\dim (\sigma _t(X)) = \min \{n,2t-1\}$ for all $t\ge 1$ (\cite{a}, Remark 1.6). Notice that $r_X(P) \ge b_X(P)$ and that equality holds at least on a non-empty open subset of $\sigma _t(X)\setminus \sigma _{t-1}(X)$, $t:= b_X(P)$.
Obviously $b_X(P) = 1$ $\Longleftrightarrow $ $P\in X$ $\Longleftrightarrow$ $r_X(P)=1$.
Hence to compute all $X$-ranks it is sufficient to compute the $X$-ranks of all points of $\mathbb {P}^n\setminus X$. In this paper we look at the case of the linearly normal elliptic curves.
We prove the following result.

\begin{theorem}\label{i1}
Let $X \subset \mathbb {P}^n$, $n \ge 3$, be a linearly normal elliptic curve. Fix $P\in \mathbb {P}^n\setminus X$ and set $w:= b_X(P)$. We have $2 \le w \le \lfloor (n+2)/2\rfloor$.
Assume $n \ge 2w+2$. Then either $r_X(P)=w$ or $r_X(P)=n+1-w$ and both cases occurs for some $P\in \sigma _w(X)\setminus
\sigma _{w-1}(X)$.
\end{theorem}

The inequalities $2 \le w \le \lfloor (n+2)/2\rfloor$ in the statement of Theorem \ref{i1} are obvious (\cite{a}, Remark 1.6). The case $w=2$ and arbitrary $n$ was settled in \cite{bgi}, Theorem
3.13.
Theorem \ref{i1} leaves partially open the cases $n = 2w$, $n= 2w-1$ and $n=2w-2$ (in which either $r_X(P) = w$ or $r_X(P) \ge n+1-w$). If $n=2w-1$, then either
$r_X(P) =w = n+1-w$ or $r_X(P) \ge w+1$ and the latter case occurs for a non-empty codimension two subset of points of $\mathbb {P}^n$ (Proposition \ref{i00}). In this case
we also have a non-trivial result on the set of all zero-dimensional schemes $Z \subset X$ evincing the border rank of the points $P$ with maximal border rank (Proposition \ref{f2}). The case $n=3$ is contained in \cite{p} (here we have $r_X(P) \le 3$ and in characteristic zero to get this inequality
it is sufficient to quote \cite{lt}, Proposition 4.1).

Following works by A. Bia\l ynicki-Birula and A. Schinzel (\cite{bs1}, \cite{bs2}), J. Jelisiejew introduced the definition of open rank for symmetric tensors, i.e. for the Veronese embeddings of projective spaces (\cite{j}). In the general case of $X$-rank we may translate the definition of open rank in the following way.

\begin{definition}\label{oo0}
Fix an integral and non-degenerate variety $X\subset \mathbb {P}^n$. For each $P\in \mathbb {P}^n$ the {\it open $X$-rank} $w_X(P)$ of $P$ is the minimal integer $t$
such that for every proper closed subset $T\subsetneq X$ there is $S\subset X$ with $\sharp (S)\le t$ and $P\in \langle S\rangle$.
\end{definition}
Obviously $w_X(P) \ge r_X(P)$, but
often the strict inequality holds (e.g., $w_X(P)>1$ for all $P$). For linearly normal elliptic curves we prove the following result.

\begin{theorem}\label{oo1}
Fix integers $w >0$ and $n \ge 2w+2$. Let $X\subset \mathbb {P}^n$ be a linearly normal elliptic curve. Fix a zero-dimensional scheme $W\subset X$ and $P\in \mathbb {P}^n$ such
that $\deg (W)=w$, $P\in  \langle W\rangle$ and $P\notin \langle W'\rangle$ for any $W'\subsetneq W$. Fix any finite set $U\subset X$. Then there is
$E\subset X\setminus U$ such that $\sharp (E) =n+1-w$ and $P\in \langle E\rangle$. There is no set $F\subset X$ such that $\sharp (F) \le n-w$, $F\cap W =\emptyset$
and $P\in \langle F\rangle$.
\end{theorem}

As an immediate corollary of Theorem \ref{oo1} we get the following result.

\begin{corollary}\label{oo2}
Fix integers $w >0$ and $n \ge 2w+2$. Let $X\subset \mathbb {P}^n$ be a linearly normal elliptic curve. Take any $P\in \sigma _w(X)\setminus \sigma _{w-1}(X)$,
i.e. with $b_X(P) =w$. Then $w_X(P) =n+1-w$.
\end{corollary}

We work over an algebraically closed field $\mathbb {K}$ such that $\mbox{char}(\mathbb {K})=0$. This assumption is essential in our proofs, mainly to quote \cite{cc1}, Proposition 5.8, which
is a very strong non-linear version of Bertini's theorem.

\section{Preliminary lemmas}\label{S2}

In this paper an elliptic curve is a smooth and connected projective curve with genus $1$.

Fix any non-degenerate variety $X\subset \mathbb {P}^n$. For any $P\in \mathbb {P}^n$ let $\mathcal {S}(X,P)$ denote
the set of all $S \subset X$ evincing $r_X(P)$, i.e. the set of all $S \subset X$ such that $\sharp (S)=r_X(P)$
and $P\in \langle S\rangle$. Notice that every $S\in \mathcal {S}(X,P)$ is linearly independent and
$P\notin \langle S'\rangle$ for any $S'\subsetneqq S$. Now assume that $X$ is a linearly normal elliptic curve.
Let $\mathcal {Z}(X,P)$ denote the set of all zero-dimensional subschemes $Z\subset X$ such that
$\deg (Z)=b_X(P)$ and $P\in \langle Z\rangle$. Lemma \ref{o1} below gives $\mathcal {Z}(X,P)\ne \emptyset$. Fix
any $Z\in \mathcal {Z}(X,P)$. Notice that $Z$ is linearly independent (i.e. $\dim (\langle Z\rangle )=\deg (Z)-1$)
and $P\notin \langle Z'\rangle$ for any subscheme $Z'\subsetneqq Z$.

\begin{notation}\label{e7}
Let $C\subset \mathbb {P}^n$ be a smooth, connected and non-degenerate
curve. Let $\beta (C)$ be the maximal integer such that every
zero-dimensional subscheme of
$C$ with degree at most $\beta (C)$ is linearly independent.
\end{notation}

The following lemma is just a reformulation of \cite{bb0}, Lemma 1.

\begin{lemma}\label{a1}
Let $Y \subset \mathbb {P}^r$ be an integral variety. Fix any $P\in \mathbb {P}^r$ and two zero-dimensional subschemes $A$, $B$ of $Y$ such that $A \ne B$, $P\in \langle A\rangle$,
$P\in \langle B\rangle$, $P\notin \langle A'\rangle$ for any $A'\subsetneqq A$ and $P\notin \langle B'\rangle$ for any $B'\subsetneqq B$.
Then $h^1(\mathbb {P}^r,\mathcal {I}_{A\cup B}(1)) >0$.
\end{lemma}

\begin{proposition}\label{e8}
Fix an integer $k \le \lfloor \beta (C)/2\rfloor$ and any $P\in \sigma _k(C)\setminus \sigma _{k-1}(C)$. Then there
exists a unique zero-dimensional scheme $Z\subset C$ such that $\deg (Z)\le k$ and $P\in \langle Z\rangle$.
Moreover $\deg (Z)=k$ and $P\notin \langle Z'\rangle$ for all $Z'\subsetneqq Z$.
\end{proposition}

\begin{proof}
The existence part is stated in \cite{bb2}, Lemma 1, which in turn is just an adaptation of some parts of the beautiful paper \cite{bgl} (\cite{bgl}, Lemma 2.1.6) or
of \cite{bgi}, Proposition 11. The uniqueness part is true
by Lemma \ref{a1} and the definition of the integer $\beta (C)$.
\end{proof}

\begin{lemma}\label{e7.1}
Let $X\subset \mathbb {P}^n$ be a linearly normal elliptic curve.

\quad (i) We have $\beta (X) =n$. A scheme $Z\subset X$ with $\deg (Z) = n+1$
is linearly independent if and only if $Z\notin \vert \mathcal {O}_X(1)\vert$. 

\quad (ii) Fix zero-dimensional schemes $A, B\subset X$ such that $\deg (A)+\deg (B)\le n+1$. If $\deg (A) +\deg (B) =n+1$ and $A+B \in |\mathcal {O}_X(1)|$,
assume $A\cap B \ne \emptyset$, i.e. assume $A\cup B \ne A+B$. Then $\langle A\rangle \cap \langle B\rangle = \langle A\cap B\rangle$.

\quad (iii) Fix zero-dimensional schemes $A, B\subset X$ such that $\mathcal {O}_X(A+B) \cong \mathcal {O}_X(1)$.
Then $\dim (\langle A\rangle \cap \langle B \rangle ) =\deg (A\cap B)$.
\end{lemma}

\begin{proof}
Let $F\subset X$ be a zero-dimensional subscheme. Since $X$ is projectively normal, we have
$h^1(\mathcal {I}_F(1)) =0$ if and only if either $\deg (F) < \deg (\mathcal {O}_X(1)) =n+1$ or $\deg (F) =n+1$
and $F\notin |\mathcal {O}_X(1)|$ (use the cohomology
of line bundles on an elliptic curve). Hence we get part (i). By the Grassmann formula we also get parts (ii) and (iii).
\end{proof}

\begin{lemma}\label{o3}
Let $X\subset \mathbb {P}^n$, $n \ge 3$, be a linearly normal elliptic curve. Fix
$P\in \mathbb {P}^n$. Then either $b_X(P) = r_X(P)$ or $r_X(P) +b_X(P)\ge n+1-b_X(P)$.
\end{lemma}

\begin{proof}
Assume $b_X(P)<r_X(P)$. Fix $W$ evincing $b_X(P)$ and $S$ evincing $r_X(P)$. Assume $\sharp (S)+\deg (W)
\le n$. Hence $S\cup W$ is linearly independent (Lemma \ref{e7.1}), i.e. $\langle S\rangle \cap
\langle W\rangle = \langle W\cap S\rangle$. Since $S$ is reduced, while $W$ is not reduced,
$W\cap S \subsetneqq W$. Hence $b_X(P) \le \deg (W\cap S)<b_X(P)$, a contradiction.
\end{proof}

\begin{lemma}\label{o1}
Let $X\subset \mathbb {P}^n$, $n \ge 3$, be a linearly normal elliptic curve. Fix a positive integer $w$
such that $2w \le n+1$. Fix $P\in \mathbb {P}^n$
and assume the existence of a zero-dimensional scheme $Z\subset X$ such that $\deg (Z)=w$, $P\in \langle Z\rangle$,
while $P\notin \langle Z'\rangle $ for all $Z'\subsetneqq Z$. Then $b_X(P)= w$.\end{lemma}

\begin{proof}
Assume $b_X(P)<w$ and take a scheme $B\in \mathcal {Z}(X,P)$ (Proposition \ref{e8}). Hence $P\in \langle B\rangle$
and $\deg (B)\le w-1$. Since $\deg (Z)+\deg (B) \le n$,
$Z\cup B$ is linearly independent. Hence $\langle Z\rangle \cap \langle B\rangle = \langle Z\cap B\rangle$.
We have $P\in \langle Z\rangle \cap \langle B\rangle$. Since $\deg (B) < w$, we have $Z\cap B\subsetneqq Z$.
Hence $P\notin \langle Z\cap B\rangle$, a contradiction. The converse part follows from Proposition \ref{e8},
part (i) of Remark \ref{o1} and the inequality $2w \le n+1$. The last assertion follows from the first part using induction
on the integer $b_X(Q)$.
\end{proof}

\section{Proof of Theorem \ref{i1} and related results}\label{S3}

\begin{proposition}\label{f2}
Fix an integer $k\ge 1$, a linearly normal elliptic curve $C\subset \mathbb {P}^{2k+1}$ and $P\in \mathbb {P}^{2k+1}\setminus \sigma _k(C)$.

\quad (a) Either $\sharp (\mathcal {Z}(C,P)) \le 2$ or $\mathcal {Z}(C,P)$ is infinite. We have $Z_1\cap Z_2=\emptyset$
and  $\mathcal {O}_C(Z_1+Z_2) \cong \mathcal {O}_C(1)$ for any $Z_1,Z_2\in \mathcal {Z}(C,P)$ such that $Z_1\ne Z_2$.

\quad (b) If $\sharp (\mathcal {Z}(C,P)) \ne 2$, then $\mathcal {O}_C(2Z) \cong \mathcal {O}_C(1)$ for all $Z\in \mathcal {Z}(C,P)$.

\quad (c) If $\mathcal {Z}(C,P)$ is infinite, then its positive-dimensional part $\Gamma$ is irreducible and one-dimensional. Fix a general $Z\in  \Gamma$. Either $Z$ is reduced or there
is an integer $m \ge 2$ such that $Z=mS_1$ for a reduced $S_1\subset C$ such that $\sharp (S_1) =(k+1)/m$.

\quad (d) If $P$ is general, then $\sharp (\mathcal {Z}(C,P))=2$.
\end{proposition}

\begin{proof}
Since no non-degenerate
curve is defective (\cite{a}, Remark 1.6), we have $\sigma _{k+1}(C)=
\mathbb {P}^{2k+1}$ and $\dim (\sigma _k(C)) =2k-1$. Hence $b_C(P) = k+1$. Proposition \ref{e8} and part (i) of Lemma \ref{e7.1} give
$\mathcal {Z}(C,P) \ne \emptyset$. Fix $Z_1, Z_2\in \mathcal {Z}(C,P)$ such that $Z_1\ne Z_2$. Parts (ii) and (iii) of Lemma \ref{e7.1}
give $\mathcal {O}_C(Z_1+Z_2) \cong \mathcal {O}_C(1)$ and $Z_1\cap Z_2 = \emptyset$, proving part (a).

\quad (i) Let $J(C,\dots
,C) \subset C^{k+1}\times
\mathbb {P}^{2k+1}$ be the abstract join of $k+1$ copies of $C$, i.e. the closure in $C^{k+1}\times \mathbb {P}^{2k+1}$
of the set of all $(P_1,\dots ,P_{k+1},P)$ such that $P_i\ne P_j$ for all $i\ne j$, the set $\{P_1,\dots,P_{k+1}\}$ is linearly
independent and $P\in \langle \{P_1,\dots ,P_{k+1}\}\rangle$. Since $\sigma _{k+1}(C)=
\mathbb {P}^{2k+1}$, for a general $P$ the set $\mathcal {Z}(C,P)$ is finite and its cardinality
is the degree of the generically finite surjection $J(C,\dots
,C)\to \mathbb {P}^{2k+1}$ induced by the projection $C^{k+1}\times
\mathbb {P}^{2k+1} \to
\mathbb {P}^{2k+1}$. Assume the existence of schemes
$Z_1, Z_2,Z_3
\in \mathcal {Z}(C,P)$ such that $Z_i\ne Z_j$ for all $i\ne j$. Part (a)
gives $Z_i\cap Z_j = \emptyset$ and $\mathcal {O}_C(Z_i+Z_j) \cong \mathcal {O}_C(1)$ for all $i\ne j$. Taking $i=1$ and $j\in \{2,3\}$
we get $\mathcal {O}_C(Z_2) \cong \mathcal {O}_C(Z_3)$. By symmetry we get $\mathcal {O}_C(Z) \cong \mathcal {O}_C(Z_1)$ for
all $Z\in \mathcal {Z}(C,P)$. Since $\mathcal {O}_C(Z_1+Z_2) \cong \mathcal {O}_C(1)$, we also get $\mathcal {O}_C(2Z)\cong \mathcal {O}_C(1)$ for
all $Z\in \mathcal {Z}(C,P)$.

\quad (ii) Now assume $\sharp (\mathcal {Z}(C,P))=1$, say $\mathcal {Z}(X,P) = \{Z\}$. Fix any $E\in \vert \mathcal {O}_C(1)(-Z)\vert$.
Since $E+Z$ is contained in a hyperplane, we have $\langle Z\rangle \cap \langle E\rangle \ne \emptyset$. Part (iii) of Lemma \ref{e7.1} gives $\dim (\langle Z\rangle \cap \langle E\rangle )=\deg (Z\cap E)$. Set $J:= \{(Q,E)\in
\langle Z\rangle \times \vert \mathcal {O}_C(1)(-Z)\vert :Q\in \langle E\rangle \}$. We just saw that $J$ is a complete
projective set. For dimensional reasons the projection of $\langle Z\rangle \times \vert \mathcal {O}_C(1)(-Z)\vert$ into its first factor induces a dominant morphism
$u: J \to \langle Z\rangle$. Since $J$ is complete, there is $E\in \vert \mathcal {O}_C(1)(-Z)\vert$
such that $u(E)=Z$. The uniqueness of $Z$ gives $E=Z$. Hence $2Z\in \vert \mathcal {O}_C(1)\vert$.
Since the set of all $Z\subset X$ such that $2Z \in \vert \mathcal {O}_C(1)\vert$ has dimension $k+1$,
we get $\sharp (\mathcal {Z}(C,P))=2$ for a general $P$, proving part (d). Since this integer, two, is the degree
of a generically finite surjection $\gamma : J(C,\dots ,C)\to \mathbb {P}^{2k+1}$ and $\mathbb {P}^{2k+1}$ is a normal variety, each fiber of $\gamma$
is either infinite or with cardinality $\le 2$. Therefore either
$\sharp (\mathcal {Z}(C,P))\le 2$ or $\mathcal {Z}(C,P)$ is infinite.

\quad (iii) Now assume that $\mathcal {Z}(C,P)$ is infinite. Since any two different elements of $\mathcal {Z}(C,P)$ are disjoint (see step (i)), for a general $A\in C$ there is at most
one element of $\Gamma$ containing $A$. Hence $\dim (\Gamma )=1$ and $\Gamma$ is irreducible. Since
a general point of $C$ is contained in a unique element of $\Gamma$, the algebraic family $\Gamma$ of effective
divisors of $C$ is a so-called {\it involution} (\cite{cc1}, \S 5). Since any two elements of $\Gamma$ are
disjoint, this involution has no base points. Let $Z$ be a general element of $\Gamma$. Either $Z$ is reduced
or there is an integer $m\ge 2$ such that $Z = mS$ with $S$ reduced (\cite{cc1}, Proposition 5.8), concluding the proof of part (c).
\end{proof}

\vspace{0.3cm}

\qquad {\emph {Proof of Theorem \ref{i1}.}} For any integer $k>0$ such that $\sigma _{k-1}(X)\ne \mathbb {P}^n$, we have
$r_X(Q) = k$ for a general $Q\in \sigma _k(X)$. Hence for arbitrary $w \le \lfloor (n+2)/2\rfloor$ there are points $P$ such that $r_X(P) = b_X(P) =w$.
Fix $w \le -1 +n/2$, $P$ and $W$ such that $b_X(P)=w$, and $r_X(P)>w$. Lemma \ref{o3} gives $r_X(P) \ge n+1-w$. Hence to prove Theorem \ref{i1} it
is sufficient to prove that $r_X(P) = n+1-w$. Fix $W\in \mathcal {Z}(X,P)$. Set
$\mathcal {B}:= \{Z+W\}_{Z\in \vert \mathcal {O}_X(1)(-2W)\vert }$. Hence $\mathcal {B} := \{B\in \vert \mathcal
{O}_X(1)(-W)\vert :W \subset B\}$. Set $\mathcal {S}:=
\{Z\in \vert \mathcal {O}_X(1)(-W)\vert :P\in \langle Z\rangle \}$. Since $\deg (\mathcal {O}_X(1)(-W)) = n+1-w \le n$,
every element of $\vert \mathcal {O}_X(1)(-W)\vert$ is linearly independent. However, in the definition
of the set $\mathcal {S}$ we did not prescribed
that $P\notin \langle Z'\rangle$ for all $Z'\subsetneqq Z$. Hence
$\mathcal {B} \subseteq \mathcal {S}$. Part (i) of
of Lemma \ref{e7.1} and the inequality $r_X(P)\ge n+1-w$ give that $r_X(P)=n+1-w$ if and only if there is a reduced $S\in
\mathcal {S}$.

\quad (a) In this step we prove that $\mathcal {B} \ne \mathcal {S}$. Fix a general subset $E\subset X$ such that
$\sharp (E) = n-2w-1$. Since $n>2w+1$, we have $E\ne \emptyset$. Hence for a general
$E$ the degree $w$ line bundles $\mathcal {O}_X(W)$ and $\mathcal {O}_X(1)(-W-E)$ are
not isomorphic. Hence to get $\mathcal {B} \ne \mathcal {S}$ it
is sufficient to prove the existence of a degree $w$ zero-dimensional subscheme $A_E$ of $X$ such that $E+A_E\in \mathcal {S}$.
Let $\ell _{\langle E\rangle }: \mathbb {P}^n\setminus \langle E\rangle\to \mathbb {P}^{2w+1}$ denote the
linear projection from $\langle E\rangle$. Call $X_E\subset \mathbb {P}^{2w+1}$ the closure of $\ell _{\langle E\rangle }\vert (X\setminus \langle E\rangle \cap X)$
in $\mathbb {P}^{2w+1}$. Since $X$ is non-degenerate, $X_E$ spans $\mathbb {P}^{2w+1}$.
Since $X$ is a smooth curve, the rational map $\ell _{\langle E\rangle }\vert (X\setminus \langle E\rangle \cap X)$
extends to a surjective morphism $\psi : X \to X_E$. Since every degree
$n-2w+1$ zero-dimensional subscheme of $X$ is linearly independent, $E$ is the scheme-theoretic intersection of $X$ with $\langle
E\rangle$. Hence $\deg
(X_E)\cdot
\deg (\psi )=deg (X) -
\deg (E) = n+1 -n+2w+1 = 2w+2$. Hence
$\deg (X_E)=2w+2$ and $\deg (\psi )=1$. Since $\deg (\psi )=1$, $X_E$ and $X$ are birational. Hence $X_E$ is a linearly normal
elliptic curve. Since $X$ and $X_E$ are smooth curves, $\psi$ is an isomorphism. Since $\langle E\rangle \cap X
=E$ (as schemes),
we have $\psi ^\ast (\mathcal {O}_{X_E}(1)) \cong \mathcal {O}_X(1)(-E)$. Set $W':= \psi (W)$. For
a general $E$ we may assume $E\cap W=\emptyset$. Hence $W'$ is a degree $w$ subscheme of $X_E$ isomorphic as
an abstract scheme to $W$. Hence $W'$ is not reduced. Fix $W_1\subsetneqq W'$ and call $W_2$ the only subscheme
of $W$ such that $\psi (W_2) = W_1$. Since $W'$ is linearly independent, $\ell _{\langle E\rangle }\vert
\langle W\rangle \to \langle W' \rangle$ is an isomorphism. Since $\ell _{\langle E\rangle }\vert W =\psi \vert W$
is an isomorphism onto $W'$ and $P\notin \langle W_2\rangle$, we get $\ell _{\langle E\rangle }(P) \notin \langle W_1\rangle$.
Since this is true for all $W_1 \subsetneqq W$, Lemma \ref{o1} gives that $W'$ evinces the border $X_E$-rank of the
point $\ell _{\langle E\rangle }(P)$. Our choice of $E$ implies $\mathcal {O}_{X_E}(2W') \ne \mathcal {O}_{X_E}(1)$.
Hence part (b) of Proposition \ref{f2} gives the existence of a unique scheme
$A\subset X_E$ such that
$A\ne W'$ and
$\ell _{\langle E\rangle}(P) \in \langle A\rangle$. Set $A_E:= \psi ^{-1}(A)$. Since $E\cap W=\emptyset$ and $\deg (A_E) =\deg
(W)$, to prove
$E+A_E\notin \mathcal {B}$ it is sufficient to prove $A_E \ne W$, i.e. (since $\psi$ is an isomorphism) $W'\ne A$. We chose
$A\ne W$. Call $X[n-2w-1]$ the set of all $E$ for which $E+A_E$ is defined. 

\quad (b) Let $\Gamma \subseteq \mathcal {S}$ be any irreducible component of $\mathcal {S}$ containing
the irreducible algebraic family $\{E+A_E\}_{E\in X[n-2w-1]}$ constructed in step (a). Let $F$ be a general element
of $\Gamma$. Remember that to prove $r_X(P)=n+1-w$ it is sufficient to find a reduced
$S\in \Gamma$. $\Gamma$ is an irreducible algebraic family of divisors of $X$. We have $\dim (\Gamma ) =n-2w-1$. By
construction for a general $E\subset X$ such that $\sharp (E)=n-2w-1$ there is $B_E\in \Gamma$ such that $E \subset B_E$. For
general $E$ we have $\langle E\rangle \cap \langle W\rangle =\emptyset$. Since $P \notin \langle E\rangle $, the scheme $\ell
_{\langle E\rangle }(W)$ is isomorphic to $W$, $P\in \langle \ell _{\langle E\rangle }(W)\rangle$ and $P\notin \langle
W'\rangle$ for any $W'\subsetneqq \ell _{\langle E\rangle }(W)$. Lemma \ref{f2} gives $\ell _{\langle E\rangle }(P)\notin
\sigma _k(X_E)$ for general $E$. For general $E$ the degree $2k+2$ line bundles $\mathcal {O}_X(2W)$ and $\mathcal
{O}_X(1)(-E)$ are not isomorphic. Hence part (b) of Proposition \ref{f2} applied to the curve $X_E$, the point $\ell _{\langle
E\rangle }(P)$ and the scheme $Z:= \ell _{\langle E\rangle }(W)$ gives that such a divisor $B_E$ is unique. Hence $\Gamma$ is an
involution in the classical terminology (\cite{cc1}, \S 5). Assume for the moment that $\Gamma$ has no fixed component. We get
that either $F$ is reduced (and hence parts (i) and (ii) of Theorem \ref{i1} are proved for $P$) or there is an integer $m\ge
2$ such that each connected component of $F$ appears with multiplicity $m$ (\cite{cc1}, Proposition 5.8). Since $F = E+A_E$
with
$E$ reduced and
$\sharp (E)>\deg (A_E)$ this is obviously false. Hence we may assume that $\Gamma$ has a base locus. Call $D$ the base locus of
$\Gamma$. Hence the irreducible algebraic family $\Gamma (-D)$ of effective divisors of $X$ has the same dimension and it is
base point free. We have
$F = D+F'$ with $F'$ general in $\Gamma (-D)$. Since $\Gamma (-D)$ is an involution without base points and whose general
member has at least one reduced connected component (a connected component of $E$), its general member $F'$ is reduced
(\cite{cc1}, Proposition 5.8). Since $D$ has finite support and $F'$ is general, we also have $F'\cap D=\emptyset$. Fix
$O\in D_{red}$. We have $O\notin \langle W\rangle$, because $\deg (W\cup \{O\}) =w+1$ and every degree $w+1$ subscheme of $X$
is linearly independent. Let $E_1$
be the union of
$O$ and
$n-2w-2$ general points of
$X$
(if $n=2w+2$, then $E_1= \{O\}$). Since $O\notin \langle W\rangle$ and $X$ is non-degenerate,
we have $\langle W\rangle \cap \langle E_1\rangle =\emptyset$. Hence the point $\ell _{\langle E_1\rangle}(P)$ is contained
in the linear span of the degree $w$ subscheme $\ell _{\langle E_1\rangle}(W)$ of the linearly normal elliptic curve
$X_{E_1}\subset \mathbb {P}^{2w+2}$, but not in the linear span of any proper subscheme of it.
Since any degree $2w+1$ subscheme of $X_{E_1}$ is linearly independent, we get
$b_{X_{E_1}}(\ell _{\langle E_1\rangle}(P)) = w+1$. Since $O$ is a base point of $\Gamma$, we also get
a one-dimensional family $\Gamma '$ of distinct degree $w+1$ subschemes of $X_{E_1}$ such that
$\ell _{\langle E_1\rangle}(P)$ is in the linear span of each of it. Part (a) of Proposition \ref{f2} gives that these
schemes are pairwise disjoint. Hence $\deg (D)=1$ and $D = \{O\}$ (as schemes). Since $E+A_E$ has at least $\deg (E_1)$
points with multiplicity one, at least one connected component of the general element $F'$ of $\Gamma '$ is reduced. Since
$F'$ is a general element of the base point free involution $\Gamma (-D)$, $F'$ is reduced (\cite{cc1}, Proposition 5.8).
Since
any degree $n$ divisor of $X$ is linearly independent, we have
$\langle E_1\rangle \cap X=E_1$ (scheme-theoretic intersection). Since $\Gamma '$ has no base points, we may also
assume that
$F' \cap (X_{E_1}\setminus \ell _{\langle E_1\rangle}(X\setminus E_1)) =\emptyset$.
Hence the counterimage $F''$ of $F'$ in $X$ is disjoint from $E_1$. Hence $F'' \cup E_1$ is reduced.
Since $P\in \langle F'' \cup E_1\rangle$, we get $r_X(P)\le n+1-w$.\qed

\begin{proposition}\label{i00}
Fix an integer $k\ge 1$ and a linearly normal elliptic curve $X\subset \mathbb {P}^{2k+1}$. Then there
are $Q, P\in \mathbb {P}^{2k+1}$ such that $b_X(Q)=b_X(P)=r_X(Q)=k+1$ and $r_X(P)\ge k+2$. The set of all such points $Q$
contains a non-empty open subset of $\mathbb {P}^{2k+1}$, while the set of all such points $P$ contains a non-empty
algebraic subset of codimension $2$ of $\mathbb {P}^{2k+1}$.
\end{proposition}

\begin{proof}
Since $\sigma _{k+1}(X) = \mathbb {P}^{2k+1}$, while $\dim (\sigma _k(X))=2k-1$ (\cite{a}, Remark 1.6), we may take as $Q$ a
general point of $\mathbb {P}^{2k+1}$.
Now we prove the existence of points
$P\in
\mathbb {P}^n$ such that
$r_X(P) >b_X(P)=k+1$ and that the set of all $P$ such that $b_X(P)=k+1<r_X(P)$ contains a codimension $2$ subset of $\mathbb
{P}^{2k+1}$. Let
$\mathcal {U}$
be the set of all degree $k+1$ schemes $Z_1 \subset X$ such that $Z_1$ is non-reduced
and $2Z_1\notin \vert \mathcal {O}_X(1)\vert$. The set $\mathcal {U}$
is a quasi-projective integral variety of dimension $k+1$.
Fix any $Z_1\in \mathcal {U}$. Let $\mathcal {V}(Z_1)$
denote the set of all non-reduced $Z_2\in \vert \mathcal {O}_X(1)(-Z_1)\vert$ such that
$Z_2\cap Z_1=\emptyset$. The set $\mathcal {V}(Z_1)$
is a quasi-projective and integral variety of dimension $k$. Since
$Z_1\cap Z_2=\emptyset$, Remark
\ref{e7.1} shows that $\langle Z_1\rangle \cap \langle
Z_2\rangle$ is a single point, $Q$. If $b_X(Q)=k+1$, then $\mathcal {Z}(X,Q) = \{Z_1,Z_2\}$,
because
$\mathcal {O}_X(2Z_1)\ne
\mathcal {O}_X(1)$ (Part (b) of Proposition \ref{f2}). Since neither $Z_1$ nor $Z_2$ is reduced, we get $r_X(Q) >k+1$.
Varying
$Z_2$ for a fixed
$Z_1$ the set of all points $Q$ obtained in this way covers a non-empty open subset of an irreducible hypersurface of $\langle
Z_1\rangle$. Assume $b_X(Q)\le k$
and fix $W\in \mathcal {Z}(X,Q)$. Notice that
$P \notin \langle W'\rangle $ for any $W'\subsetneqq W$. Since $\deg (W)+\deg (Z_1) \le n$, Lemma \ref{a1}
and Lemma \ref{e7.1} give the existence of $Z'\subsetneqq Z$ such that $Q\in \langle Z'\rangle$. Iterating
the trick taking $Z'$ and $W$ instead of $Z_1$ and $W$ we get $W\subseteqq Z'$ and hence $W\subset Z_1$
Making this
construction using $Z_2$ and $W$ we get $W\subsetneqq Z_2$. Since $Z_1\cap Z_2=\emptyset$, we obtained a contradiction.
\end{proof}

\section{Proof of Theorem \ref{oo1}}

\qquad {\emph {Proof of Theorem \ref{oo1}:}} Since $X$ is linearly normal, for any zero-dimensional scheme $Z\subset X$ we have $h^1(\mathbb {P}^n,\mathcal {I}_Z(1)) = h^1(X,\mathcal {O}_X(1)(-Z))$.
Hence $h^1(\mathbb {P}^n,\mathcal {I}_Z(1)) >0$ if and only if either $\deg (Z) \ge n+2$ or $\deg (Z) =n+1$ and $Z\in |\mathcal {O}_X(1)|$. Assume the existence of the set $F$. Since $F\cap W =\emptyset$ and $P\in \langle W\rangle \cap \langle F\rangle$, we have $h^1(\mathcal {I}_{W\cup F}(1)) >0$ (Lemma \ref{a1}). Since $\deg (F\cup W)\le n$, we
get a contradiction. Define $\mathcal {B}$ and $\mathcal {S}$ as in the proof of Theorem \ref{i1}. Step (a) of the quoted proof works with no modification. At the end
of that step we defined $X[n-2w-1]$ and now we continue the proof of Theorem \ref{oo1} in the following way.

\quad (b) Let $\Gamma \subseteq \mathcal {S}$ be any irreducible component of $\mathcal {S}$ containing
the irreducible algebraic family $\{E+A_E\}_{E\in X[n-2w-1]}$. To prove Theorem \ref{oo1} for the point $P$ it is sufficient to find a reduced
$S\in \Gamma$ such that $S\cap U =\emptyset$. $\Gamma$ is an irreducible algebraic family of divisors of $X$. We have $\dim (\Gamma ) =n-2w-1$. By
construction for a general $E\subset X$ such that $\sharp (E)=n-2w-1$ there is $B_E\in \Gamma$ such that $E \subset B_E$. For
a general $E$ we have $\langle E\rangle \cap \langle W\rangle =\emptyset$. Since $P \notin \langle E\rangle $, the scheme $\ell
_{\langle E\rangle }(W)$ is isomorphic to $W$, $P\in \langle \ell _{\langle E\rangle }(W)\rangle$ and $P\notin \langle
W'\rangle$ for any $W'\subsetneqq \ell _{\langle E\rangle }(W)$. Lemma \ref{e7.1} gives $\ell _{\langle E\rangle }(P)\notin
\sigma _{w-1}(X_E)$ for a general $E$. For general $E$ the degree $2w$ line bundles $\mathcal {O}_X(2W)$ and $\mathcal
{O}_X(1)(-E)$ are not isomorphic. Hence Proposition \ref{e8} and part (i) Lemma \ref{e7.1} applied to the curve $X_E$, the point $\ell _{\langle
E\rangle }(P)$ and the scheme $Z:= \ell _{\langle E\rangle }(W)$ gives that such a divisor $B_E$ is unique. Thus $\Gamma$ is an
involution in the classical terminology (\cite{cc1}, \S 5).

\quad (b1) In this step we assume that $\Gamma$ has no base points. Since $\Gamma$ has no base points and $U$ is a fixed finite set, a general $S\in \Gamma$ is contained in $X\setminus U$. Hence
to conclude the proof of Theorem \ref{oo1} it is sufficient to prove that a general $S\in \Gamma$ is reduced. Either $S$ is reduced or there is an integer $m\ge
2$ such that each connected component of $S$ appears with multiplicity $m$ (\cite{cc1}, Proposition 5.8). Since $F = E+A_E$
with
$E$ reduced and
$\sharp (E)>\deg (A_E)$ this is obviously false. 

\quad (b2) Assume that $\Gamma$ has a base locus. Call $D$ the base locus of
$\Gamma$. Thus the irreducible algebraic family $\Gamma (-D)$ of effective divisors of $X$ has the same dimension and it is
base point free. We have
$F = D+F'$ with $F'$ general in $\Gamma (-D)$. Since $\Gamma (-D)$ is an involution without base points and whose general
member has at least one reduced connected component (a connected component of $E$), its general member $F'$ is reduced
(\cite{cc1}, Proposition 5.8). Since $D$ has finite support and $F'$ is general, we also have $F'\cap D=\emptyset$. Fix
$O\in D_{red}$. We have $O\notin \langle W\rangle$, because $\deg (W\cup \{O\}) =w+1$ and every degree $w+1$ subscheme of $X$
is linearly independent. Let $E_1$
be the union of
$O$ and
$n-2w-2$ general point of
$X$
(if $n=2w+2$, then $E_1= \{O\}$). Since $O\notin \langle W\rangle$ and $X$ is non-degenerate,
we have $\langle W\rangle \cap \langle E_1\rangle =\emptyset$. Thus the point $\ell _{\langle E_1\rangle}(P)$ is contained
in the linear span of the degree $w$ subscheme $\ell _{\langle E_1\rangle}(W)$ of the linearly normal elliptic curve
$X_{E_1}\subset \mathbb {P}^{2w+2}$, but not in the linear span of any proper subscheme of it.
Since any degree $2w+1$ subscheme of $X_{E_1}$ is linearly independent, we get
$b_{X_{E_1}}(\ell _{\langle E_1\rangle}(P)) = w+1$. Since $O$ is a base point of $\Gamma$, we also get
a one-dimensional family $\Gamma '$ of distinct degree $w+1$ subschemes of $X_{E_1}$ such that
$\ell _{\langle E_1\rangle}(P)$ is in the linear span of each of it. Part (a) of Proposition \ref{f2} gives that these
schemes are pairwise disjoint. Hence $\deg (D)=1$ and $D = \{O\}$ (as schemes). Since $E+A_E$ has at least $\deg (E_1)$
points with multiplicity one, at least one connected component of the general element $F'$ of $\Gamma '$ is reduced. Since
$F'$ is a general element of the base point free involution $\Gamma (-D)$, $F'$ is reduced (\cite{cc1}, Proposition 5.8).
Since
any degree $n$ divisor of $X$ is linearly independent, we have
$\langle E_1\rangle \cap X=E_1$ (scheme-theoretic intersection). Since $\Gamma '$ has no base points, we may also
assume that
$F' \cap (X_{E_1}\setminus \ell _{\langle E_1\rangle}(X\setminus E_1)) =\emptyset$.
Hence the counterimage $F''$ of $F'$ in $X$ is disjoint from $E_1$. Thus $F'' \cup E_1$ is reduced.
Since $P\in \langle F'' \cup E_1\rangle$, we get $r_X(P)\le n+1-w$. 

Steps (b1) and (b2) conclude the proof of Theorem \ref{oo1}.\qed

\providecommand{\bysame}{\leavevmode\hbox to3em{\hrulefill}\thinspace}

\end{document}